\documentclass{amsart}

\usepackage{amsmath,amssymb,amsfonts}
\usepackage[english]{babel}
\usepackage{amsxtra}
\usepackage{mathtools}
\usepackage[normalem]{ulem}
\usepackage[all,cmtip]{xy}
\usepackage[colorlinks=true]{hyperref}
\usepackage{pgf,tikz-cd}
\usetikzlibrary{arrows}

\title[Conic bir]{Remarks on the group of birational selfmaps of a conic fibration}

\author{Enrica Floris}

\newtheorem{theorem}{Theorem}[section]
\newtheorem{lemma}[theorem]{Lemma}

\newtheorem{corollary}[theorem]{Corollary}
\theoremstyle{definition}
\newtheorem{definition}{Definition}[section]

\newtheorem{notass}[definition]{Notation/Assumption}
\theoremstyle{remark}

\numberwithin{equation}{section}

\newcommand{\IC}{\mathbb C}

\newcommand{\IP}{\mathbb P}

\newcommand{\IZ}{\mathbb Z}

\newcommand{\cO}{\mathcal{O}}

\newcommand{\orth}{\mathrm{O}}
\newcommand{\GL}{\mathrm{GL}}
\newcommand{\Hilb}{\mathrm{Hilb}}

\DeclareMathOperator{\Aut}{Aut}

\DeclareMathOperator{\Bir}{Bir}
\DeclareMathOperator{\Autz}{Aut^{\circ}}

\thanks{EF thanks the Institut Universitaire de France for the excellent working conditions. She warmly thanks J\'er\'emy Blanc and Thomas Dedieu for many useful conversations, Hsueh-Yung Lin and Christian Urech for comments on a previous version of this note and J\'anos Koll\'ar for comments and for suggesting the statement of Theorem \ref{thm:spinor} and some references.}

\begin{document}

\maketitle

\begin{abstract}
We study the group of birational selfmaps of a variety birational to a conic bundle.
We prove that it admits a surjective morphism to the direct sum of an uncountable number of copies of $\mathbb Z/2\IZ$.
\end{abstract}

\section{Introduction}
The group of birational self-maps of a projective variety is a very rich birational invariant.
If $X$ is a rationally connected variety, group-theoretic properties of $\Bir(X)$ should mirror the geometry of $X$.
One striking instance of this, is the recent result in \cite{RUvS25a}: if a variety $X$ is such that $\Bir(X)\cong \Bir(\IP^n)$, then $X$ is a rational $n$-dimensional variety, see
 \cite{CERUvS} for more results in this direction.
In this note, we focus on groups of birational selfmaps of varieties birational to conic fibrations.

\smallskip

The aim of this note is to prove the following result.

\begin{theorem}\label{thm:Intro}
Let $X, Y$ be projective varieties, let $f\colon X\to Y$ be a fibration whose general fibre 
is a rational curve. Then there is an uncountable set $I$ and a surjective homomorphism $\Phi\colon\Bir(X/Y)\to\bigoplus_I\mathbb Z/2\mathbb Z$.
\end{theorem}

Theorem \ref{thm:Intro} is proved essentially in \cite{BLZ},
where infinitely many involutions of a conic bundle were produced, and those are distinguished by their indeterminacy locus $\Gamma$. Our key observation is that we can find uncountably many such $\Gamma$.
We prove that in section \eqref{sec:thmIntro}.

The homomorphism $\Phi$ has been defined in \cite[Theorem B and D]{BLZ} in the following way.
The set $I$ parametrises equivalence classes of marked conic bundles $(X/Y,\Gamma)$ defined in \cite[Definition 3.22]{BLZ} where $\Gamma\subseteq Y$ is a divisor.
Let $\tau \colon X\dasharrow X$ be a type II link 
$$
\xymatrix{
X'\ar[d]_{\varepsilon}\ar@{-->}[r]& X''\ar[d]^{\varepsilon'}\\
X\ar[d]_f &X\ar[d]^f\\
Y\ar@{=}[r]&Y
}
$$
where $\varepsilon$ is the blow up of a codimension 2 subvariety $Z$ such that $f\vert_Z$ has degree 1 onto the image. Then $\Phi(\tau)$ is the sequence having a 1 in correspondence with the class of $(X/Y, f(Z))$ and zeroes elsewhere.

Recall that for a conic fibration $f\colon X\to Y$ the group $\Bir(X/Y)$ is isomorphic to $\Aut(X_\eta)$ where $X_\eta$ is the generic fibre of $f$, and $\Aut(X_\eta)$ is an orthogonal group over $\mathbb C(Y)$, where $\mathbb C(Y)$ is the field of rational functions of $Y$.
In this note we also give an interpretation of $\Phi$ in terms of the spinor norm homomorphism (see \cite[Section 55]{OMeara73}). 

\begin{theorem}\label{thm:spinor}
Let $X, Y$ be projective varieties, let $f\colon X\to Y$ be a fibration whose general fibre 
is a rational curve. 
Let $\Phi$ be the homomorphism of \cite[Theorem B and D]{BLZ} and Theorem \ref{thm:Intro}.
Then there is a factorization $\Phi=\xi\circ\overline{\mathrm{div}}\circ \theta$
where \begin{itemize}
\item $\overline{\theta}\colon \Bir(X/Y)\to  \mathbb C(Y)^*/(\mathbb C(Y)^*)^2$ is the spinor norm homomorphism;
\item if $\mathcal P$ is the set of prime divisors of $Y$, then $\overline{\mathrm{div}}\colon \mathbb C(Y)^*/(\mathbb C(Y)^*)^2\to \bigoplus_{\mathcal P}\mathbb Z/2\mathbb Z$ sends the class of a rational map $\varphi$ with $\mathrm{div}(\varphi)=\sum_{P\in\mathcal P} n_P P$ to $\sum_{P\in\mathcal P} \bar n_P P$ where for an integer $n$, we denote by $ \bar n$ its class modulo 2;
\item $\xi\colon \bigoplus_{\mathcal P}\mathbb Z/2\mathbb Z\to \bigoplus_I\mathbb Z/2\mathbb Z$ is the projection to the quotient.
\end{itemize}
\end{theorem}

Theorem \ref{thm:Intro} implies in particular that the group $\Bir(X)$ is uncountable.
The uncountability follows much more directly from the  fact that $\Bir(X/Y)$ is isomorphic to an orthogonal group over $\mathbb C(Y)$.

\smallskip

Conic bundles are one of the outcomes of the MMP applied to a variety with non-pseudoeffective canonical divisor and there are several examples of rationally connected varieties which are birational to a conic bundle but are not rational, for example the cubic threefold and the stably rational, but not rational, threefold of \cite{BCTSSD85}.

By a result of Rosenlicht (see \cite[Proposition 1.9]{FFZ} and \cite[Proposition 2.5.1]{BFT22}),  if $X$ is a rationally connected and non-rational variety of dimension 3, then $\Autz(X)$ is trivial. It might be tempting to conjecture that in this case $\Bir(X)$ is discrete, but by the uncountability of the group of birational selfmaps of a conic bundle this is not true and we have the following.

\begin{corollary}
Let $X$ be a rationally connected projective variety of dimension 3. Assume that
$X$ is birational to a conic bundle and $X$ is not rational.
Then $\Bir(X)$ is uncountable and for every variety $Y$ birational to $X$ we have $\Autz(Y)=\{id_Y\}$.
\end{corollary}

In \cite{BF13} and \cite{RUvS25b} a study of the topology on $\Bir(\IP^n)$ and $\Bir(X)$ was carried. We revisit \cite[Theorem B]{BLZ} and Theorem \ref{thm:Intro} from this topological point of view and get the following analog of the results in Hudson and Pan \cite{Pan99}.

\begin{corollary}\label{cor:ubdd gen}
Let $X, Y$ be projective varieties, let $f\colon X\to Y$ be a fibration whose general fibre 
is a rational curve. Assume that $\dim Y\geq 2$ Then there is no set of generators of $\Bir(X)$ of bounded degree.

More precisely, for every very ample line bundle $A$ on $Y$ there is an integer $d_0$ such that the generators of $\Bir(X)$ include a countable (infinite) number of algebraic families parametrised by global sections of $dA$ for $d\geq d_0$.
\end{corollary}

\section{Preliminaries}
We work over $\IC$. For a field $F$ we denote by $F^*$ the set $F\setminus\{0\}$ with the structure of a multiplicative group.

\subsection{Birational involutions of conic bundles}\label{sseq:involutions}
In \cite{BLZ} the following involutions of conic bundles were defined.

Let Y be a smooth variety, $\eta\colon P\to Y$ a locally trivial $\IP^2$-bundle,
and $X\subseteq P$ a closed subvariety such that the restriction $f$ of $\eta$ is a conic bundle. Let $s\colon Y\dasharrow P$ be a rational section  whose image is not contained in $X$. 
We set $\tau_s\in\Bir(X/Y)$ the birational involution whose restriction to a general fibre $\eta^{-1}(y)$
is the involution induced by the projection from $s(y)$. Let $\Gamma\subseteq Y$ be an irreducible
hypersurface not contained in the discriminant locus of $f$, and let $q$ be a local
equation of $X$ in $P$.

The induced birational map $\tau_s$ is a type II link, obtained as the blow-up of $s(Y)\cap X$ and the contraction of the strict transform of $f^{-1}(f(s(Y)\cap X))$.

\subsection{The orthogonal group}
We review some very classical material on the orthogonal group.
Our reference is the book by O'Meara \cite{OMeara73}.

Let $F$ be a field, $V$ a finite-dimensional vector space over $F$, $B$ a bilinear form and $Q$ the associated quadratic form.

We say that the vector space $V$ is \textit{regular} if we have $$\{y\in V\vert\;\forall x\in V\;B(x,y)=0\}=\{0\}.$$

The \textit{orthogonal group} of a regular vector space $V,Q$ is defined as follows
$$\orth(V)=\{g\in \GL(V)\vert\;\forall x\in V\;Q(x)=Q(gx)\}.$$
We denote by $\Omega(V)$ its commutator and by $\orth^+(V)$ the kernel of the determinant.
Let $v\in V$ be such that $Q(v)\neq 0$. The \textit{symmetry with respect to the vector $v$ } is defined as $$\sigma_v(x)=x-\frac{2B(x,v)}{Q(v)}v.$$
We notice that if $\lambda\in F$ is a non-zero element, then $\sigma_v=\sigma_{\lambda v}$.
Every element of $\orth(V)$ is the product of at most $\dim V$ symmetries \cite[Theorem 43:3]{OMeara73}.
Moreover, if $\sigma_{v_1}\circ\ldots\circ\sigma_{v_\ell}$ and $\sigma_{w_1}\circ\ldots\circ\sigma_{w_m}$ are two factorizations of the same element, then there is $\alpha\in F$ such that $$Q(v_1)\ldots Q(v_\ell)=\alpha^2 Q(w_1)\ldots Q(w_m).$$

The \textit{spinor norm} is the following group homomorphism
$$
\begin{array}{rcc}
\theta\colon \orth(V)&\longrightarrow& F^*/(F^*)^2\\
g&\mapsto& \overline{Q(y_1)\ldots Q(y_\ell)}
\end{array}
$$
where $g=\sigma_{y_1}\circ\ldots\circ\sigma_{y_\ell}$ and $\overline{\beta}$ denotes the class of $\beta\in F^*$ in the quotient.
If $\dim V\in\{1,2,3\}$, then the kernel of $\theta$ coincides with the commutator $\Omega(V)$ (see \cite[Proposition 55:5]{OMeara73}).

\subsection{Topology on $\Bir(X)$}\label{ssec:top}
Let $X$ be a complex projective variety and let us fix a closed embedding $X\subseteq \IP^N$.
Then for every $d\geq 1$ by \cite[Definition 4.1, Lemma 4.2]{RUvS25b} there is 
a set $H_d$ that is locally closed in a projective variety together with a morphism $\pi_d\colon H_d\to \Bir(X)$ with the following properties
\begin{enumerate}
\item $\displaystyle \Bir(X)=\cup_{d\geq1}\pi_d(H_d)$, and $\Bir(X)$ is endowed with the inductive limit topology with respect to the filtration $\pi_i(H_i)\subseteq\pi_{i+1}(H_{i+1})$;
\item for every variety $A$ with a morphism $\rho\colon A\to \Bir(X)$ there is an open affine covering $\{A_i\}_{i\in I}$ and for every $i$ an integer $d_i$ and a morphism $\rho_i\colon A_i\to H_{d_i}$ such that $\displaystyle \rho\vert_{A_i}=\pi_{d_i}\circ \rho_i$.
\end{enumerate}

\begin{definition}
We say that a subset $\mathcal S\subseteq\Bir(X)$ has bounded degree id there is an integer $D$ such that $\mathcal S\subseteq \cup_{d=1}^D\pi_d(H_d)$.
\end{definition}

Moreover, the identification of a birational self-map with its graph 
gives an identification of $\Bir(X)$ with an open subset of $\Hilb(X\times X)$ (see \cite[Proposition 1.7]{Han87}) and for every polynomial $P$ we denote by $\Hilb_P$ the intersection of the connected component of the Hilbert scheme corresponding to $P$ with $\Bir(X)$.
Then, by \cite[Corollary 3.4]{RUvS25b} the inclusion $\Hilb_P\to \Bir(X)$ is a morphism and the sets $\Hilb_P$ give a partition of $\Bir(X)$.

By (2) above, a subset $\mathcal S$ has bounded degree if and only if there is a finite set of Hilbert polynomials $P_1,\ldots,P_\ell$ such that $\mathcal S\subseteq\cup_i \Hilb_{P_i}$.

\section{Proof of Theorem \ref{thm:Intro}}\label{sec:thmIntro}

We start by proving some preparatory lemmas. We refer to \cite{KM} for the basic notions on the MMP.

\medskip

The following lemma is a higher-dimensional version of the initial discussion in \cite{Serrano}. We prove that a (birationally) isotrivial fibration whose fibres are of general type is essentially a product.
The result is standard, but we include a proof here for lack of an appropriate reference.
\begin{lemma}\label{lem:isotr}
Let $f\colon Z\to C$ be a fibration from a smooth variety $Z$ to a curve $C$.
Assume that the general fibres of  $f$ are of general type  and birational to each other.
Let $F$ be the canonical model of a general fibre of $f$.
Then there is a curve $B$ and a group $G$ acting faithfully on $F$ and $B$
such that $C\cong B/G$ and  $Z$ is birational to $(F\times B)/G$ over $C$.
\end{lemma}
\begin{proof}
We perform an MMP relative over $C$, which terminates with a good minimal model by \cite{BCHM}. The canonical divisor on the endproduct of the MMP is thus relatively semiample and induces a contraction to a fibration $f'\colon Z'\to X$ where $K_{Z'}$ is relatively ample. After passing to $Z'$ and by uniqueness of the canonical model  , we can assume that $f$ is isotrivial and the general fibre is isomorphic to $F$.
By \cite[Theorem 2.11]{Kol87}  there is a finite map $\tau\colon\widetilde C\to C$ such that $Z\times_C\widetilde C$ is birational to $F\times \widetilde C$ over $\widetilde C$.

Since the canonical divisor of the general fibre is ample, the map $Z\times_C\widetilde C\dasharrow F\times \widetilde C$ is an isomorphism over an open set $\widetilde C_0\subseteq\widetilde C$. 
Let $C_0\subseteq C$ be an open set such that $\tau^{-1}C_0\subseteq\widetilde C_0$ and such that $\tau\colon \widetilde C_0\to C_0$ is \'etale.
Therefore, if $\Delta\subseteq C_0$ is a disc, the preimage $f^{-1}\Delta$ is isomorphic to $F\times \Delta$.
Moreover, if $\overline C_0$ is the universal cover of $C_0$,
the base change $Z_0\times_{C_0}\overline C_0$ is isomorphic to $F\times \overline{C_0}$, where $Z_0=f^{-1}(C_0)$.
Indeed, there is a factorization $\overline C_0 \to \widetilde C_0\to C_0$ and an isomorphism
$$Z_0\times_{C_0}\overline C_0 \cong(Z_0\times_{C_0}\widetilde{C_0})\times_{\widetilde{C_0}}\overline{C_0}\cong(F\times \widetilde C_0)\times_{\widetilde{C_0}}\overline{C_0}\cong F\times \overline C_0.$$
Moreover, there is an action of $\pi_1(C_0)$ on $F\times \overline C_0$ such that
$Z_0\cong (F\times \overline C_0)/\pi_1(C_0)$ and the induced action on $F$ is given by the monodromy.

\smallskip

With the local trivializations, we define a monodromy homomorphism $\rho\colon\pi_1(C_0)\to\Aut(F)$. Let $G$ be the image of $\rho$.
The group $G$ is finite because $F$ is of general type. The subgroup $\ker \rho$ has finite index in $\pi_1(C_0)$ and therefore corresponds to an \'etale cover $B_0\to C_0$ where $\pi_1(B_0)$ identifies with $\ker\rho$ (see for example \cite[Riemann existence theorem]{Milne}).

\noindent We denote by $Z_{B_0}$ the base-change $Z_0\times_{C_0}B_0$.
There is an action of $G\cong\pi_1(C_0)/\pi_1(B_0)$ on $Z_{B_0}$ such that $Z_0\cong Z_{B_0}/G$. Moreover,
$$Z_{B_0}\cong (F\times \overline C_0)/\pi_1(B_0)\cong (F\times \overline C_0)/\ker\rho\cong F\times (\overline C_0/\ker\rho)\cong F\times B_0.$$
We conclude that $Z_0\cong(F\times B_0)/G$.
The curve $B_0$ compactifies to a projective curve $B$ and, thus,
the variety $Z$ is birational to $(F\times B)/G$.
The group $G$ acts faithfully on $F$ because $G\subseteq \Aut(F)$ and on $B$ by the construction of $B$ as finite cover.
\end{proof}

In the next lemma, we prove that general complete intersections of codimension 2 inside a conic bundle are not birational.

\begin{lemma}\label{lem:nonisotr}
Let $Y$ be a smooth variety of dimension at least two embedded in a projective space $\mathbb P^m$. We denote by $\cO_Y(n)$ the restriction $\cO_{\IP^m}(n)\vert_Y$.
Let $a,b$ be positive integers and $\mathcal E=\cO_Y\oplus\cO_Y(a)\oplus\cO_Y(b)$.
Let $X\subseteq \IP_Y(\mathcal E)$ be such that the general fiber of the restriction $g=f\vert_X$ of the projection $f\colon \IP_Y(\mathcal E)\to Y$ is a rational curve. Let $A$ be a very ample divisor on $Y$.
Then there are two integers $\bar d_1,\bar d_2$ such that for every $d_i\geq \bar d_i$
if $H$ is such that $\cO_X(H)\sim \cO_{}(1)\otimes g^*\cO_Y(d_1 A)$ and
 $H_1\in |H|$ is a general section,
 two general sections $\Gamma, \Gamma'\in |(H+d_2 g^*A)\vert_{H_1}|$ are not birational.

\end{lemma}

\begin{proof}
Let $\mu\colon W\to X$ be a resolution of singularities and denote by $E$ the divisor of discrepancies, such that $$K_W=\mu^*K_X+E.$$
The divisor $\mu^*(H+kf^*A)$ is base point free and big for every $k> 0$.
Therefore, the general member of $|\mu^*(H+kf^*A)|$ is smooth and connected by the Bertini theorem for irreductibility~\cite{Jouanolou}, the Seidenberg theorem~\cite{Seidenberg}  and the generalized Seidenberg theorem~\cite[Theorem~1.7.1]{BS} and its proof. 
We fix  $H_1\in |H|$ and denote its pullback, which coincides with its strict transform, by $\widetilde H_1$.
We notice that the restriction of $g\circ\mu$ to $\widetilde H_1$ is generically finite.
By adjunction, we get 
$$K_X\sim(\cO(-1)\otimes g^*\cO_Y(D))\vert_X$$
for a suitable divisor $D$ on $Y$, and
$$K_{\widetilde H_1}=\mu^*(K_X\otimes\cO_X(1)\otimes g^*\cO_Y(kA)\otimes\cO_X(E))\vert_{\widetilde H_1}\sim\left(\mu^*g^*\cO_Y(D+kA)\otimes\cO_X(E)\right)\vert_{\widetilde H_1}.$$
This divisor is big if $k$ is big enough.
Let $\bar d_1$ be such that $K_{\widetilde H_1}$ is big.
Let $d\geq0$ and $H_{2,d}\in |H+dg^*A|$ and $\Gamma_d=H_1\cap H_{2,d}$.
Then the restriction $g\vert_{\Gamma_d}$ is birational onto its image.
We denote by $\widetilde H_{2,d}$ the strict transform of $H_{2,d}$ in $W$, which coincides with its pullback
and by $\widetilde \Gamma_d$ the intersection $\widetilde H_1\cap \widetilde H_{2,d}$.
Let $d_0$ be a positive integer such that for every $d\geq d_0$ the Kodaira dimension of $\mu^*(K_X+2H)+E+d\mu^*g^*(A)$ equals $\dim(Y)$.
By the adjunction formula, if $d\geq d_0$, the divisor $$K_{\widetilde \Gamma_d}\sim\mu^*(K_X+2H)\vert_{\widetilde \Gamma_d}+E\vert_{\widetilde \Gamma_d}+d\mu^*g^*(A)\vert_{\widetilde \Gamma_d}$$
is big.
Assume by contradiction that two general elements in $|\mu^*(H+d f^*A)\vert_{\widetilde H_1}|$ are birational to each other.
Let $\widetilde \Gamma_{d},\widetilde \Gamma'_{d}\in |\mu^*(H+d f^*A)\vert_{\widetilde H_1}|$ be two general elements and let $Z\to \widetilde H_1$ be a resolution of the pencil.
Let $\overline{\Gamma}$ be the canonical model of $\widetilde \Gamma_{d}$.
By Lemma \ref{lem:isotr}, there is a group $G$ and a curve $C$ such that $Z$ is birationally equivalent to $(\overline{\Gamma}\times C)/G$ and $G$ acts on $\overline{\Gamma}\times C$ diagonally and on $\overline{\Gamma}$ and $C$ faithfully.
Since $\widetilde H_1$, and therefore $Z$, is of general type, both $\Gamma$ and $C$ are of general type.
Since $G$ acts faithfully on both $\overline{\Gamma}$ and $C$, the projection morphism $\pi\colon \overline{\Gamma}\times C\to(\overline{\Gamma}\times C)/G$ is \'etale in codimension 1.

Hence, we have $K_{\Gamma\times C}=\pi^*K_{(\Gamma\times C)/G}$ and $K_{(\Gamma\times C)/G}$ is ample.
This implies that $(\Gamma\times C)/G$ is the canonical model of $\widetilde H_1$. Therefore, there is a map $h\colon \widetilde H_1\dasharrow (\Gamma\times C)/G$ which is a composition of divisorial contractions, flips and a final morphism from the endproduct of the MMP to the canonical model.
If $\Delta$ is a big divisor on $\widetilde H_1$, then $h_*\Delta$ is big.
On the other hand, we have $h_*\widetilde \Gamma_d=[\Gamma\times\{x\}]$ for some $x$
for a general member of the pencil.
This is a contradiction.

\end{proof}

\begin{lemma}\label{lem:deviss}
Let $X, Y$ be projective varieties and $X\to Y$ be a fibration whose general fibre is a rational curve.
Let $a,b$ be positive integers.
Then there exist a positive integer $m$, a variety $X'$ birational to $X$ and $Y'$ smooth birational to $Y$ and a diagram
$$
\xymatrix{
X\ar@{-->}[r]^{\alpha}\ar[d]&X'\ar@{^{(}->}[r]^<(.2){\gamma}\ar[d]&\IP_{\IP^m}(\cO\oplus\cO(a)\oplus\cO(b))\ar[d]\\
Y\ar@{-->}[r]^{\beta}&Y'\ar@{^{(}->}[r]^{\delta}&\IP^m
}
$$ 
where $\alpha,\beta$ are birational and $\gamma,\delta$ are embeddings.
\end{lemma}
\begin{proof}
Let $Y'\to Y$ be a desingularisation, $X''$ the normalisation of $X\times_Y Y'$ and $f\colon X''\to Y'$
the induced fibration.
The variety $Y$ is projective, thus it is embedded in a projective space $\IP^m$.
The sheaf $f_*\omega_{X''/Y'}$ is a rank 3 vector bundle over an open set $U_1$ of $Y'$.
Moreover, if $X_1=f^{-1}U_1$, then we have an embedding of $X_1$ in the projectivisation of $f_*\omega_{X''/Y'}$
$$
\xymatrix{
X_1\ar[rd]\ar@{^{(}->}[rr]&&\IP_{U_1}(f_*\omega_{X_1/U_1})\ar[ld]\\
&U_1&
}
$$
After possibly shrinking $U_1$, we can assume that $f_*\omega_{X_1/U_1}$ is a trivial vector bundle.
Let $U_2$ be a trivializing open set for $\cO(a)\vert_{Y'}$ and $\cO(b)\vert_{Y'}$.

Set $U=U_1\cap U_2$ and $X_0=f^{-1}U$.
We fix two isomorphisms $\varphi\colon \IP_{U}(f_*\omega_{X/U})\to \IP^2\times U$ and $\psi\colon \IP_{U}(\cO_U\oplus\cO(a)\vert_U\oplus\cO(b)\vert_U)\to \IP^2\times U$.
We have thus
$$
\xymatrix{
X_0\ar@{^{(}->}[r]^<(.2)i\ar[rd]&\IP_{U}(f_*\omega_{X/U})\ar[r]^<(.1){\psi^{-1}\circ\varphi}\ar[d]&\IP_{U}(\cO_U\oplus\cO(a)\vert_U\oplus\cO(b)\vert_U)\ar[ld]\\
&U&
}
$$
Finally, we set $X'$ to be the Zariski closure of  $\psi^{-1}\circ\varphi\circ i(X_0)$ in $\IP_{Y'}(\cO_{Y'}\oplus\cO(a)\vert_{Y'}\oplus\cO(b)\vert_{Y'})$.
\end{proof}

The following is Theorem \ref{thm:Intro}.

\begin{theorem}
Let $X, Y$ be projective varieties, let $f\colon X\to Y$ be a fibration whose general fibre 
is a rational curve.
 Then there exists an uncountable set $I$ and a group homomorphism
 $$\Bir(X/Y)\to\bigoplus_I\mathbb Z/2.$$
\end{theorem}

\begin{proof}
By Lemma \ref{lem:deviss}, we can assume that $X\to Y$ decomposable in the sense of \cite[Definition 6.8]{BLZ}.
We follow the proof of \cite[Proposition 6.9]{BLZ} where the group homomorphism is defined.
We only need to prove that the index set $I$ is uncountable.
By Lemma \ref{lem:nonisotr} there are two integers $\bar d_1,\bar d_2$ such that for every $d_i\geq \bar d_i$
if $H$ is such that $\cO_X(H)\sim \cO_{}(1)\otimes f^*\cO_Y(d_1 A)$ and
 $H_1\in |H|$ is a general section,
 two general sections $\Gamma, \Gamma'\in |(H+d_2 f^*A)\vert_{H_1}|$ are not birational.
 We fix such $d_1,d_2$.

In the proof of \cite[Proposition 6.9]{BLZ}, 
instead of chosing two sections in the same linear system $| H_0 +dF|$,
we chose $H_1\in| H_0 +d_1F|$ and $H_2\in| H_0 +d_2F|$ with 
we chose thus $d\geq m_0$.
Therefore, if  $\Gamma, \Gamma'$ are two complete intersections of the form $H_1\cap H_2$ with $H_2\in| H_0 +d_2F|$ general, then $\Gamma, \Gamma'$ are not birational and the models $(X/Y,f(\Gamma))$, $(X/Y,f(\Gamma'))$ are not in the same equivalence class in the sense of \cite[Definition 2.12]{BLZ}.
Moreover, their classes have non-vanishing image in the direct sum $\displaystyle \oplus_{M(X/Y)}\mathbb Z/2$.
\end{proof}

We are also ready to prove Corollary \ref{cor:ubdd gen}.

\begin{proof}[Proof of Corollary \ref{cor:ubdd gen}]

By Lemma \ref{lem:deviss} we can assume that $X$ is embedded in a projective bundle and the fibration is the restriction of the natural fibration.

\medskip

Let $P$ be a Hilbert polynomial such that $\Hilb_P$ defined in \eqref{ssec:top} is non empty.
Let $u\colon\mathcal U_P\to \Hilb_P$ be the universal family and $p_i\colon \mathcal U_P\to X$ the natural maps.
Then $(p_2,u)\colon\mathcal U_P\to X\times \Hilb_P$ is a birational morphism.
Let $\mathbb E$ be the exceptional locus of $(p_2,u)$.
Notice that for every $\varphi\in \Hilb_P$ the fiber  $\mathbb E_\varphi$ overt $\varphi$ is the exceptional locus of $\Gamma_{\varphi}\to X$ where  $\Gamma_{\varphi}$ denotes the graph of $\varphi$ and contains the strict transforms of all the divisors contracted by $\varphi$.
We set $d(P)$ as
$$\sup\left\{
\begin{array}{l|l}
& \text{ there is }\varphi\in \Hilb_P,\\
covgon(\Gamma)&\; \text{ there is } E\text{ irreducible component of }\mathbb E_\varphi,\;\\
&E\text{ birational to }\IP^1\times\Gamma
\end{array}
 \right\}.$$
Since $\Hilb_P$ is quasi-projective, the sup is a max.
By \cite[Proposition 6.9]{BLZ} and its proof, since $X$ carries a conic fibration structure, the set of integers $d(p)$ is unbounded.
This proves that there is no set of generators of bounded degree.

\smallskip

By Lemma \ref{lem:nonisotr} and \cite[Proposition 6.9]{BLZ} there are two integers $\bar d_1, \bar d_2$ such that for any $d_i\geq \bar d_i$, letting 
$H$ be such that $\cO_X(H)\sim \cO_{}(1)\otimes f^*\cO_Y(d_1 A)$ and
 $H_1\in |H|$  a general section,
for every  $\Gamma\in |(H+d_2 g^*A)\vert_{H_1}|$ there is an ivolution $\sigma_{\Gamma}\in \Bir(X/Y)$ contracting a divisor birational to $\Gamma\times\IP^1$.
 
Moreover, two general sections $\Gamma, \Gamma'\in |(H+d_2 g^*A)\vert_{H_1}|$ are not birational and if $d_i$ is high enough $\Gamma\times\IP^1$ and $\Gamma'\times\IP^1$ are not birational. This proves the second part of the statement.
\end{proof}

\section{The automorphism group of a quadric and the relative bir of a conic bundle}\label{sec:spinor}

Let $F$ be a field and
let $Z$ be a smooth quadric in $\IP^n(F)$, let $Q=0$ be an equation for $X$.
Then $F^n$ with the quadratic form $Q$ is a regular vector space.
The orthogonal group surjects onto
the automorphism group of $Z$ and acts by projectivities:
$$
\begin{array}{rcc}
\orth(F^n)\times \IP^n(F)&\longrightarrow&\IP^n(F)\\
(g,[x])&\mapsto&[g(x)].
\end{array}
$$

The action above induces a surjective homomorphism

$$\Pi\colon\orth(F^n)\to\Aut(Z).$$

Throughout this section we will use the following notation and make the following assumptions.
\begin{notass}
Let $X,Y$ be normal varieties and $f\colon X\to Y$ a conic bundle. 
There is a rank 3 vector bundle $\mathcal E$ such that $X\subseteq \IP(\mathcal E)$ and $f$ is the restriction of the natural morphism  $\pi\colon\mathcal E\to \IP(\mathcal E)$.
Set $F=\IC(Y)$.
\end{notass}

\medskip

Then, the fibre $X_{\eta}$ over the generic point $\eta$ of $Y$ is a smooth quadric in $\IP^2(F)$.
The group $\Bir(X/Y)$ is isomorphic to the automorphism group $\Aut(X_\eta)$ of the generic fibre, and therefore there is a surjective homomorphism
$\orth(F^3)\to\Bir(X/Y)$.

Let $\sigma\in \orth(V)$ be a symmetry.
Then there is $v\in F^3$ such that $Q(v)\neq 0$ such that
$$\sigma([x])=\left[x-\frac{2B(x,v)}{Q(v)}v\right]=[Q(v)x-2B(x,v)v].$$

\begin{lemma}\label{lem:symmetries}
%Let $X,Y$ be normal varieties and $f\colon X\to Y$ a conic bundle.
%Assume that there is a rank 3 vector bundle $\mathcal E$ such that $X\subseteq \IP(\mathcal E)$ and $f$ is the restriction of the natural morphism.
%Let $F=\IC(Y)$ and 
Let $\sigma\in\orth(F^3)$ be a symmetry.
Then there is a rational section $s\colon Y\dasharrow \mathcal E$ such that $\sigma=\tau_s$ where $\tau_s$ is defined in \eqref{sseq:involutions}.
\end{lemma}
\begin{proof}
Let $U$ be an affine open set, and $\pi\colon\mathcal E\to \IP(\mathcal E)$ the natural projection. Let $Q$ be an equation of $\pi^{-1}X$ over $U$.
If $\sigma$ is a symmetry, then there is $v\in F^3$, $Q(v)\neq 0$ such that $\sigma=\sigma_v$. The element $v$ defines a rational section $s\colon Y\dasharrow \mathcal E$ whose image is not contained in  $\pi^{-1}X$ because $Q(v)\neq 0$.

If we set $x_1 \in X$, $\sigma$ is defined at $x_1$ and $x_2=\sigma(x_1)$, then it is easy to check, using the definition of $\sigma$, that
the intersection of $X$ with the line through $x_1$ and $\pi(v)$ is the set $\{x_1,x_2\}$. Thus $\sigma$ is the automorphism covering of the $2:1$ map induced by the projection from $s$. It coincides with $\tau_s$ defined in \eqref{sseq:involutions}.
\end{proof}

\subsection{Birational involutions and the spinor map}
%Let $f\colon X\to Y$ be a conic bundle.
%Assume that there is a rank three vector bundle $\mathcal E$ such that 
%$$
%\xymatrix{
%X\ar[rd]\ar[rr]&&\IP(\mathcal E)\ar[ld]\\
%&Y&
%}
%$$

Over an affine set $\mathrm{Spec}A\cong U\subseteq Y$ where $\mathcal E$ is trivial, the field $F$ is the quotient field of $A$.
Let $Q$ be a local equation of $\pi^{-1}X$.
After possibly shrinking $U$, we may assume that $Q$ defines a morphism $Q\colon \mathcal E\vert_U\to \IC$.
An element $v\in F^3$ such that $Q(v)\neq 0$ induces a birational section
$s\colon Y\dasharrow \mathcal E$ such that $s(Y)$ is not contained in $\pi^{-1}X$.

\smallskip

Then, for every section $s\colon U\dasharrow \mathcal E\vert_U$
there is a rational function $Q\circ s$ on $Y$ obtained by composing with the bilinear form $Q\colon\mathcal E\vert_U\cong U\times \IC^3\to\IC$.
%Moreover, if $y_1, y_2 \in \mathcal E\vert_U$
%induce the same birational section of $\IP(\mathcal E)$, then
% there is $\lambda\in F$ non zero, with $s_{y_2}=\lambda s_{ y_1}$.
% In this case $Q\circ s_{y_2}=\lambda^2 Q\circ s_{y_1}$.
% Thus, for a birational section $s_{[y]}\colon Y\dasharrow \IP(\mathcal E)$,
% the class $\overline{Q\circ \tilde s_y}\in \stackrel{.}{F}/(\stackrel{.}{F})^2$ is well defined,
% where $\tilde s_y\colon Y\dasharrow \mathcal E$ is any lift of $s_{[y]}$ over an open set
% and $\overline{\beta}$ denotes the class of $\beta\in \stackrel{.}{F}$ in the quotient.
 
 By combining Lemma \ref{lem:symmetries} with the previous discussion, we get the following
 \begin{lemma}\label{lem:spinorII}
% Let $X,Y$ be normal varieties and $f\colon X\to Y$ a conic bundle.
%Assume that there is a rank 3 vector bundle $\mathcal E$ such that $X\subseteq \IP(\mathcal E)$ and $f$ is the restriction of the natural morphism.
Let $Q$ be a local equation of $\pi^{-1}X$ in $\mathcal E$.
%where $\pi\colon\IP(\mathcal E)\to \mathcal E$.
Let  $\sigma\in\orth(F^3)$ be a symmetry.
Let $s$ be a rational section of $\mathcal E$ be such that $\sigma=\tau_s$.
Then the image of $\sigma$ under the spinor map coincides with the class of $Q\circ s$ in $F^*/(F^*)^2$.
 \end{lemma}

\subsection{A homomorphism to a direct sum of $\IZ/2\IZ$}\label{sec:div}
 
We compute now the principal divisor $\mathrm{div}(Q\circ s)$ associated to the rational function $Q\circ s$.\\

We will make the further assumption that there is an embedding $Y\subseteq \IP^m$
and that $\mathcal E$ is the restriction to $Y$ of the vector bundle
$$\mathcal F=\cO_{\IP^m}\oplus\cO_{\IP^m}(a)\oplus\cO_{\IP^m}(b)$$
to $Y$ for $a,b$ non negative integers.\\
We see $\mathcal F$ as the quotient of $\mathbb A^3\times \mathbb A^{m+1}\setminus\{0\}$
by the action of $\mathbb G_m$
$$\mu\cdot(x_0,x_1,x_2,y_0,\ldots,y_m)=(x_0,\mu^{-a}x_1,\mu^{-b}x_2,\mu y_0,\ldots,\mu y_m).$$
A compactification of $\mathcal F$ is given by
and $\overline{\mathcal F}=\IP(\cO_{\IP^m}\oplus\cO_{\IP^m}\oplus\cO_{\IP^m}(a)\oplus\cO_{\IP^m}(b))$, which we see
as the quotient of $\mathbb A^4\setminus\{0\}\times \mathbb A^{m+1}\setminus\{0\}$
by the action of $\mathbb G_m^2$
$$(\lambda,\mu)\cdot( v,x_0, x_1, x_2,y_0,\ldots,y_m)=(\lambda v,\lambda x_0,\lambda\mu^{-a}x_1,\lambda\mu^{-b}x_2,\mu y_0,\ldots,\mu y_m).$$

The variety $X$ is thus defined in $\IP(\mathcal E)$
by an equation of the form
$$Q(x,y)=\sum_{0\leq i\leq j\leq 2}\alpha_{i,j}(y)x_i x_j=0$$
where 

$$
\begin{array}{lll}
\alpha_{0,0}\in H^0(Y,\mathcal O_Y)&\alpha_{0,1}\in H^0(Y,\mathcal O_Y(a))&\alpha_{0,2}\in H^0(Y,\mathcal O_Y(b))\\
\alpha_{1,1}\in H^0(Y,\mathcal O_Y(2a))&\alpha_{1,2}\in H^0(Y,\mathcal O_Y(a+b))&\alpha_{2,2}\in H^0(Y,\mathcal O_Y(2b)).
\end{array}
$$

We notice that $Q$ is homogeneous with respect to the action of $\mathbb G_m$, thus it defines a morphism $Q\colon \mathcal F\to \IC$.
Moreover, we have a morphism

$$
\begin{array}{rccc}
\bar Q\colon&\overline{\mathcal F}&\to&\IP^1\\
&[v:x_0:x_1:x_2;y_0:\ldots:y_m]&\mapsto&[v^2:Q(x,y)].
\end{array}
$$

Let $U\subseteq Y$ be the maximal open set such that $s_{y}\colon U\to \mathcal E\vert_U$ is a regular section.

Therefore $Q\circ s_{y}$ is regular on $U$ and the support of $\mathrm{div}(Q\circ s_{y})\vert_U$ is the locus $s_y\left(s_y(Y)\cap \{Q=0\}\right)$.

Set $\overline{\mathcal E }=\IP_Y(\cO_Y\oplus\cO_Y\oplus\cO_Y(a)\oplus\cO_Y(b))$.
Then $\overline{\mathcal E }$ is a compactification of $\mathcal E$ and
we set $\Delta=\overline{\mathcal E }\setminus \mathcal E$.

Set $\overline{s(Y)}$  the Zariski closure of $s(Y)$ in $\overline{\mathcal E }$ and $g\colon \overline{s(Y)}\to Y$ the restriction of the projection.
Then  the intersection $\Delta\cap \overline{s(Y)}$ has dimension equal to $\dim Y-1$. Its projection in $Y$ coincides with the poles of $Q\circ s$ over an open subset $Y_0$ of $Y$ whose complement has codimension at least two in $Y$.
More precisely, 
\begin{equation}\label{eq:div}
\mathrm{div}(Q\circ s)\vert_{Y_0}=g_*\left(\bar Q^*(0)\right)-g_*\left(\bar Q^*(\infty)\right)=g\left(s(Y)\cap \{Q=0\}\right)-2g_*(\Delta).
\end{equation}

\bigskip 
 
% Let $\mathcal P$ be the set of prime divisors in $Y$.
%There is a homomorphism $$\Phi\colon\stackrel{.}{F}/(\stackrel{.}{F})^2\longrightarrow \bigoplus_{\mathcal P}\IZ/2\IZ$$ 
%sending the class of $\varphi\in\IC(Y)\setminus\{0\}$ with divisor $\mathrm{div}(\varphi)=\sum_{P\in\mathcal P}n_P P$ to the sequence $(\varepsilon(P))_{P\in\mathcal P}$
%where $\varepsilon(P)$ is the class of $n_P$ mod 2.

%The image of $Q\circ s_y$ via $\Phi$ is the set $s_y(Y)\cap \{Q=0\}$.

%Lemma 6.7. Let B be a smooth variety, ˆη : P B a locally trivial P2-bundle,
%and X ⊂ P a closed subvariety such that the restriction of ˆη is a conic bundle
%η : X B. Let s: B P be a rational section (i.e. a rational map, birational
%to its image, such that ˆη ◦ s = idB), whose image is not contained in X. Let
%ι ∈ Bir(X/B) be the birational involution whose restriction to a general fibre η−1(b)
%is the involution induced by the projection from s(b). Let Γ ⊂ B be an irreducible
%hypersurface not contained in the discriminant locus of η, and let F be a local
%equation of X in P.
%If the multiplicity of F(s) along Γ is equal to 0 or 1, then the parity of ι along Γ
% (in the sense of Definition 6.1) is equal to this multiplicity (modulo 2).

\subsection{Proof of Theorem \ref{thm:spinor}}
We are ready to prove Theorem \ref{thm:spinor}.
\begin{proof}
Since symmetries generate $\Bir(X/Y)$, it is enough to prove that $\Phi$ and $\xi\circ\overline{\mathrm{div}}\circ\theta$ coincide on a symmetry $\sigma$.\\
By Lemma \ref{lem:deviss}, we can assume that $Y$ is embedded in $\IP^m$, $X$ in $\IP_Y(\cO_Y\oplus\cO_Y(a)\oplus\cO_Y(b))$ and the conic fibration is the restriction of the projection $\IP_Y(\cO_Y\oplus\cO_Y(a)\oplus\cO_Y(b)) \to Y$.
We set $\pi\colon \cO_Y\oplus\cO_Y(a)\oplus\cO_Y(b)\to \IP_Y(\cO_Y\oplus\cO_Y(a)\oplus\cO_Y(b))$.\\
Let $Q=0$ be a local equation of the preimage of $\pi^{-1}X$ in  $\cO_Y\oplus\cO_Y(a)\oplus\cO_Y(b)$. 

\medskip

\noindent By Lemma \ref{lem:symmetries} there is a section $s\colon Y\dasharrow\cO_Y\oplus\cO_Y(a)\oplus\cO_Y(b) $ such that $\sigma=\tau_s$.\\
By Lemma \ref{lem:spinorII} the image of $\tau_s$ via the spinor map is the class of $Q\circ s$ in $\IC(Y)^*/(\IC(Y)^*)^2$.\\
By \eqref{sec:div} and \eqref{eq:div}, we have
$$\overline{\mathrm{div}}(Q\circ s)=f_*(\pi_*\left((Q=0)\cap \overline{s(Y)}\right)).$$

By \cite[Lemma 6.7]{BLZ} the image of this divisor with coefficients in $\IZ/2\IZ$ coincides with the image of $\tau_s$ via $\Phi$.

\end{proof}

\bibliographystyle{alpha}
\bibliography{biblioUBir}

\end{document}